\documentclass[11pt,english,reqno]{article}%{amsart}

\usepackage{amsmath, amsthm, amssymb}  %all AMS (amssymb loads amsfonts)
\usepackage{mathrsfs}
\usepackage{mathabx,bbm}

\usepackage{babel}
\numberwithin{equation}{section}

{\theoremstyle{plain}

\newtheorem{lemma}{Lemma}[section]
\newtheorem{theorem}[lemma]{Theorem}
\newtheorem{corollary}[lemma]{Corollary}
\newtheorem{proposition}[lemma]{Proposition}

} {\theoremstyle{definition}
\newtheorem{definition}[lemma]{Definition}
\newtheorem{example}[lemma]{Example}

}

\renewcommand{\mathbb}{\mathbbm}                     % use mathbbm
\renewcommand{\epsilon}{\varepsilon}                 % AMS symbols
\renewcommand{\phi}{\varphi}
\renewcommand{\theta}{\vartheta}
\renewcommand{\le}{\leqslant}
\renewcommand{\ge}{\geqslant}
\newcommand{\origsetminus}{} \let\origsetminus=\setminus     % redefine setminus
\renewcommand{\setminus}{\!\origsetminus\!}
\newcommand{\origfoo}{} \let\origfoo=\sqrt           % redefine square root
\renewcommand{\sqrt}[1]{\origfoo{#1}\;}

\newcommand{\abs}[1]{\left\lvert #1 \right\rvert}    % absolut value
\newcommand{\norm}[1]{\left\lVert #1 \right\rVert}   % norm
\newcommand{\Q}{\mathbb Q}                           % expectation
\DeclareMathOperator{\F}{{\mathcal F}}                   % sigma-fiel
\DeclareMathOperator{\R}{{\mathbb R}}                % reals
\DeclareMathOperator{\Rp}{{\mathbb R}_+}             % positive reals
\DeclareMathOperator{\N}{{\mathbb N}}                % integer
\DeclareMathOperator{\Id}{ Id}                        % identity
\DeclareMathOperator{\I}{\mathcal I}                          %

\newcommand{\A}{{\mathcal A}}

\DeclareMathOperator{\Borel}{{\mathcal B}}
\newcommand{\scapro}[2]{\langle #1,#2\rangle}       %Skalarprodukt
\DeclareMathOperator{\1}{\mathbbm 1}
\renewcommand{\S}{{\mathcal S}}
\renewcommand{\L}{{\mathcal L}}

\newcounter{zahl}

\numberwithin{equation}{section}

\newcommand{\bcase}{\begin{cases}}
\newcommand{\ecase}{\end{cases}}
\newcommand{\pmat}{\begin{pmatrix}}
\newcommand{\epmat}{\end{pmatrix}}

\newcommand{\barray}{\begin{array}{rcl}}
\newcommand{\earray}{\end{array}}

\newcommand{\del}[1]{}

\newcommand{\be} {\begin{enumerate} }
\newcommand{\ee} {\end{enumerate} }

\newcommand{\TT}{{\rm I \kern -0.2em T}}

\newcommand{\DEQS}{\begin{eqnarray*}}
\newcommand{\EEQS}{\end{eqnarray*}}
\newcommand{\DEQSZ}{\begin{eqnarray}}
\newcommand{\EEQSZ}{\end{eqnarray}}

\DeclareMathOperator{\M}{\mathcal M}          % \space of probabity measures
\DeclareMathOperator{\ii}{i}                  % isomorphism Hilbert and l^2
\DeclareMathOperator{\jj}{j}                 % inverse of \ii

\begin{document}

\title{Copulas in Hilbert spaces}

\author{
Erika Hausenblas\\
Department of Mathematics\\ Montanuniversity Leoben\\ 8700 Leoben\\
 Austria
\and
Markus Riedle\footnote{markus.riedle@kcl.ac.uk}{}\\
Department of Mathematics\\
King's College\\
London WC2R 2LS\\
United Kingdom}

\maketitle

\begin{abstract}
In this article, the concept of copulas is generalised to infinite dimensional  Hilbert spaces. We show one direction of Sklar's theorem and explain that the other direction
fails in infinite dimensional Hilbert spaces. We derive a necessary and sufficient condition which allows to state this direction of Sklar's theorem in Hilbert spaces. We consider copulas with densities and specifically construct copulas in a Hilbert space by a family of pairwise copulas with densities.
\end{abstract}

\noindent
{\rm \bf AMS 2010 subject classification:}  62H20, 60E05, 28D20, 94A17\\
{\rm \bf Key words and phrases:} copulas, Gaussian copulas, pairwise copulas,
Hilbert spaces, cylindrical measures

 \thispagestyle{empty}

\section{Introduction}

{Correlation} is the most widely applied measure for dependence between random variables. However, it is well known that correlation reflects the dependence structure only in very specific situations well. Most obviously, correlation fails to capture any nonlinear dependencies in a data set. If the underlying distribution
is not normal or even not elliptic, then the failure of correlation as a  measure for dependence is demonstrated in Embrechts, McNeil and Straumann \cite{EmbrechtsMcNeilStraumann}. Applying correlation as a measure for dependence
requires finite second moments for the underlying distribution and its estimation
often relies on finite higher moments. As a consequence, its usage becomes problematic
or must be even ruled out if the underlying distribution exhibits heavy tails
as it is often the case for financial data, see e.g. Cont \cite{Cont2001}.

There exist various alternatives to the correlation as a measure of dependence,
such as  {\em Hellinger distance}, {\em Kendall's tau}, {\em Kullback-Leibler divergence}  and the {\em mutual information}. Most prominent example is the approach by {\em copulas}. Although copulas were accused of  increasing the severity of the financial crisis 2007-08, see Salmon \cite{Salmon}, they are still widely applied in  quantitative finance and insurance industry for risk management and pricing; see e.g. McNeil et al.\
\cite{McNeil-etal} and Malevergne and Sornette \cite{MalevergneSornette} for pre-crisis publications and Brigo et al.\ \cite{Brigo} for a post-crisis publication.  But the application of copulas are not restricted to quantitative finance and actuarial sciences. In hydrology, copulas are  used to model dependencies among storm variables or to interpolate precipitation across spaces; see AghaKouchak et al.\  \cite{AghaKouchak}, Golian et al.\ \cite{Rainfall1} and Salvadori and De Michele \cite{hydro1}. Copulas are applied in meteorology, see Bonazzi et al.\ \cite{windstorm} and Fuenteset et al.\ \cite{maxtemp} and geophysics, see Yu et al. \cite{geophysics}.
Copulas can be found for modelling travel behavior in Bhat and Elurua \cite{trans1} and Eluru et al. \cite{trans3}. Copulas are used in reliability engineering for analysing
the dependence of components of complex systems in Lai and Xie \cite{LaiXie}. Numerous more applications of copulas can be listed in medicine, environmental and civil engineering
and many other areas.

Despite their popularity, copulas are seen as a strictly finite dimensional concept.
The reason  might be found in the fact that their definition strongly relies on properties
of cumulative distribution functions, which are a much less useful tool in infinite dimensional spaces. In this article, we show that nevertheless the concept of copulas can be generalised to infinite dimensional Hilbert spaces. We demonstrate this novel concept by several examples for constructing copulas in Hilbert spaces.

Our original motivation of this work is to develop a tool for describing the dependence  of the components of an infinite dimensional L\'evy process. The majority of specific examples of L\'evy processes in infinite dimensional spaces are constructed by sums of independent components, see e.g.\ Priola and Zabczyk \cite{PriolaZabczyk}. For finite dimensional L\'evy processes, describing the dependence structure by copulas is a well developed tool, resulting in {\em L\'evy copulas}; see e.g.\ \cite{Barndorff-NielsenLindner} and \cite{KallsenTankov}. In the infinite dimensional setting, we will introduce this concept in a forthcoming work
based on the present article.

Our generalisation of copulas enables us to model  the dependence structure of complex dynamical behaviour in time and space by copulas as it is well known for ordinary systems. For illustrating this, let $V(t,x)$ present a stochastic process in time $t$ and space $x$, e.g.\ the solution of a stochastic partial differential equation. By assuming that $V(t,\cdot)$ is an element of a Hilbert space with orthonormal basis
$(e_k)_{k\in\N}$, the stochastic process $V$ can be represented by
its Fourier expansion as
$V(t,x)=\sum_{k=1}^\infty Y_k(t,x)e_k(x)$ where $Y_k(t,x):=\scapro{V(t,x)}{e_k(x)}$. The dependence structure
of the Hilbert space valued random variables $\{Y_k(t,\cdot):\, k\in\N\}$ for
some $t\ge 0$
can be described by the copulas introduced in the present  work. As a specific example,
one might think of $V(t,x)$ as the displacement of a oil rig from its idle state at $x\in [0,L]$
at time $t$ caused by heavy waves. One way to model this rig is by an Euler--Bernoulli beam of length $L$ exposed to a random process. As an orthonormal basis $\{e_k:k\in\N\}$
of  the space $L^2([0,L])$ of square-integrable functions one can choose the
Fourier basis, i.e.\ the modes of the beam subjected to the boundary conditions.
In this example, the copula gives information on the dependence of the displacement on the amplitude of the wave. If the beam starts to vibrate in one mode, knowing the copula, one has some information how probable the other modes  will be effected and how quick the resonance will propagate to other modes. In particular, one might forecast some crack or other damages.

We introduce our notation and shortly review copulas in the next Section
\ref{se.preliminary}. Our approach to copulas in Hilbert spaces is
described in Section \ref{se.Hilbert} where we also state and prove one
direction of Sklar's theorem. The failure of the other direction in Sklar's theorem in Hilbert spaces is  explained in the following Section \ref{se.product}. Here we also derive a necessary and sufficient condition which allows to state this direction of Sklar's theorem.
In the last Section \ref{se.densities}, we  consider copulas with densities and specifically construct copulas in a Hilbert space bya family of pairwise copulas
with densities.

\section{Preliminaries}\label{se.preliminary}

Let $H$ be a separable Hilbert space with inner product $\scapro{\cdot}{\cdot}$,
induced norm $\norm{\cdot}$ and an orthonormal basis $(e_k)_{k\in\N}$.
The projection on  the Euclidean space is defined for $k\le \ell$ by
\begin{align*}
\pi_{e_k,\dots  ,e_\ell}\colon H \to \R^{\ell-k+1},\qquad
\pi_{e_k,\dots, e_\ell}(h):=\big( \scapro{h}{e_k},\dots,\scapro{h}{e_\ell}\big).
\end{align*}
The Borel $\sigma$-algebra is denoted by $\Borel(H)$. The function
\begin{align}\label{eq.H-and-l2}
\ii\colon H\to \ell^2, \qquad
\ii(h)=\big(\scapro{h}{e_k}\big)_{k\in\N}.
\end{align}
maps the Hilbert space $H$ isometrically isomorphic to the space $\ell^2$ of square summable
sequences.

If $(S,{\mathcal S})$ is a measurable space then $\M(S)$ denotes the
space of all probability measures on ${\mathcal S}$. Most often, we consider $\M(H)$ for a Hilbert space $H$ with the Borel $\sigma$-algebra $\Borel(H)$, in which case we equip $\M(H)$ with the Prokhorov metric. The subspace of probability measures $\mu$ with $p$-th moments, i.e.\ $\int \norm{u}^p\,\mu(du)<\infty$, is denoted with $\M^p(H)$. The space of
product measures is denoted by $\M^{\otimes}(H)$, that is
\begin{align*}
 \M^{\otimes}(H):=\left\{\mu\in \M(H):\,
 \mu\circ \ii^{-1}=\bigotimes_{k=1}^\infty \mu\circ \pi_{e_k}^{-1} \right\}.
\end{align*}
The space $\M^{c}(H)$ denotes the space of all
probability measures $\mu\in\M(H)$ such that all marginal measures have continuous
cumulative distribution function (cdf), that is
\begin{align*}
 \M^{c}(H):=\left\{\mu\in \M(H):\,
x\mapsto \mu\circ \pi_{e_k}^{-1}\big((-\infty,x]\big)
\text{ is continuous for all }k\in\N \right\}.
\end{align*}
We often use intersections of some of these subspaces, e.g.\ $\M^{\otimes}\cap\M^2(H)$.

For a probability measure $\mu\in\M(\R^k)$ the {\em cumulative distribution function (cdf)} is defined as the function
\begin{align*}
 F_\mu\colon\R^k\to [0,1],
 \qquad F_\mu(x_1,\dots, x_k):=
 \mu\big(I(x_1)\times \cdots \times I(x_k)\big),
\end{align*}
where $I(x):=(-\infty,x]$ for all $x\in\R$. We denote the set of all such intervals by
\begin{align*}
\I:=\big\{ (-\infty,x]:\, x\in\R\big\}.
\end{align*}

A {\em copula in $\R^k$} for some $k\in\N$ is a cumulative distribution function $C_k\colon [0,1]^k\to [0,1]$ with uniform marginal distributions, that is
\begin{align*}
C_k(1,\dots, 1, u_j,1,\dots, 1)= u_j\qquad
\text{for all }u_j \in [0,1],\, j\in \{1,\dots, k\}.
\end{align*}
We say that a copula $C_k\colon [0,1]^k\to [0,1]$ has a density if there exists a measurable function $c_k\colon [0,1]^k\to \Rp$ satisfying
\begin{align*}
  C_k(u_1,\dots, u_n)=\int_0^{u_1}\dots \int_0^{u_k} c_k(v_1,\dots, v_n)\,dv_n\cdots dv_1
  \qquad\text{for all }u_1,\dots, u_n\in [0,1].
\end{align*}

Sklar's Theorem (see \cite[Th.3.2.1]{MalevergneSornette} or \cite[Th.5.3]{McNeil-etal}) guarantees that for every probability measure $\mu\in
\M(\R^k)$  there exists a copula $C_k$ in $\R^k$ such that
$$
\mu \Big(I_1\times \,\cdots\,\times  I_k\Big)=C_k\Big(\mu_1\big(I_1\big),\ldots, \mu_k\big(I_k\big)\Big)
\qquad\text{for all }I_1,\dots, I_k\in\I,
$$
where $\mu_j:=\mu\circ \pi_{e_j}^{-1}$. If each marginal distribution $\mu_j$ has a continuous cdf, i.e.\ $\mu\in
\M^c(\R^k)$, then the copula $C_k$ is unique.
The converse statement of Sklar's Theorem, which is much easier to prove,
states that for a given copula $C_k$ in $\R^k$ and every product probability measure $\mu\in  \M^{\otimes}(\R^k)$ there exists a measure $\nu\in \M(\R^k)$
satisfying
\begin{align}\label{eq.mu-and-nu}
\nu \Big(I_1\times \,\cdots\,\times  I_k\Big)=C_k\Big(\mu_1\big(I_1\big),\ldots, \mu_k\big(I_k\big)\Big)
\qquad\text{for all }I_1,\dots, I_k\in\I,
\end{align}
where $\mu_j:=\mu\circ \pi_{e_j}^{-1}$. As the left hand side is the cdf of $\nu$, the measure $\nu$ is uniquely determined
by the copula $C_k$ and the product probability measure $\mu$. Thus, a copula $C_k$ in $\R^k$ induces a mapping
\begin{align*}
\hat{C}_k\colon \M^{\otimes}(\R^k)\to  \M(\R^k),
 \qquad \mu\mapsto \hat{C}_k(\mu),
\end{align*}
where $\hat{C}_k(\mu)$ denotes the probability measure $\nu$ satisfying \eqref{eq.mu-and-nu}.

\section{Copulas in Hilbert spaces}\label{se.Hilbert}

As copulas are strongly based on cumulative distribution functions it is rather a finite dimensional concept.
In order to extend it to an infinite dimensional seting, we begin with a minimal requirement.
\begin{definition}
A family $(C_k)_{k\in \N}$ of copulas $C_k$ in $\R^k$ is called {\em consistent}
if it obeys for all $k\in\N$ that
  \begin{align*}
    C_{k+1}(u_1,\dots, u_k,1)= C_k(u_1,\dots, u_k) \qquad \text{for all }
    u_1,\dots, u_k\in [0,1].
\end{align*}
\end{definition}
The following theorem shows that a consistent family of copulas
is sufficient to generalise one direction of Sklar's theorem to the infinite dimensional setting.
\begin{theorem}\label{th.Sklar-infinite}(Sklar's Theorem)\hfill\\
For each probability measure $\mu\in \M(H)$  there exists a family of consistent copulas $(C_k)_{k\in \N}$  satisfying for all  $I_1,\dots, I_k\in \I$ and $k\in\N$:
\begin{align}\label{eq.copula-equ}
  \mu\circ\pi_{e_1,\dots, e_k}^{-1}\Big(I_1\times \cdots \times I_k\Big)
  = C_k\Big(\mu_1(I_1),\dots, \mu_k(I_k)\Big),
\end{align}
where $\mu_j=\mu\circ\pi_{e_j}^{-1}$.
If $\mu\in \M^c(H)$ then the family of consistent copulas $(C_k)_{k\in \N}$ is unique.
\end{theorem}
\begin{proof}
Sklar's Theorem in finite dimensions guarantees that for each $k\in\N$ there exists a copula $C_k$ in $\R^k$ satisfying \eqref{eq.copula-equ}.
For each $k\in\N$ and $I_1,\dots, I_k\in\I$ we have
\begin{align*}
  \mu\circ \pi_{e_1,\dots, e_k,e_{k+1}}^{-1}(I_1\times\dots \times I_k\times\R)=
    \mu\circ \pi_{e_1,\dots, e_k}^{-1}(I_1\times\dots \times I_k).
\end{align*}
Consequently, it follows from \eqref{eq.copula-equ} that for all
\begin{align*}
(u_1,\dots, u_k)\in R_k:=
\bigtimes_{j=1}^{k} \big\{\mu\circ \pi_{e_j}^{-1}(-\infty,x]:\, x\in\R\big\}
\end{align*}
one obtains
$  C_{k+1}(u_1,\dots, u_k,1)= C_k(u_1,\dots, u_k) $. If $u_j\notin
 \big\{\mu\circ \pi_{e_j}^{-1}(-\infty,x]:\, x\in\R \big\}$
for some $j\in\N$ then define for each $k\ge j$
\begin{align*}
C_k(u_1,\dots, u_{j-1},u_j,u_{j+1},\dots, u_k)
:= C_k(u_1,\dots, u_{j-1},\tilde{u}_j,u_{j+1},\dots, u_k),
\end{align*}
where $\tilde{u}_j:=\sup\{\mu\circ \pi_{e_j}^{-1}(-\infty, x):\,
\mu\circ \pi_{e_j}^{-1}(-\infty, x)<u_j, x\in\R\}$.
\end{proof}

Applying Theorem \ref{th.Sklar-infinite} to the finite dimensional situation in $\R^n$  shows that each
copula $C_n$ in $\R^n$ defines a consistent family $(C_k)_{k=1,\dots, n}$. Before we study the converse of Sklar's Theorem we present some recipes to construct a family of consistent copulas in a Hilbert space.
\begin{lemma}\label{le.copulas-with-2}
  Let $\phi_j\colon [0,1]^2\to\Rp$ be the continuous density of a copula in $\R^2$ for each $j\in\N$. Then,
  by defining for every $u_1,\dots, u_k\in [0,1]$ and $k\in\N\setminus\{1\}$ the function
\begin{align}\label{eq.def-consistent-2}
  C_k(u_1,\dots, u_k)=\int_0^{u_1}\cdots \int_0^{u_k}
   \phi_1(v_1,v_2)\phi_2(v_2,v_3)\cdots \phi_{k-1}(v_{k-1},v_k)\, dv_k\cdots dv_1
\end{align}
one obtains a family $(C_k)_{k\in\N}$ of consistent copulas.
\end{lemma}
\begin{proof}
Let $\Phi$ denote a copula in $\R^2$ with continuous density $\phi$ with partial derivatives $\Phi_u$ and $\Phi_v$ with respect to the first and second argument, respectively. As $\phi$ is continuous the second derivatives satisfy $\Phi_{u,v}=\Phi_{v,u}=\phi$. It follows for each $u\in [0,1]$ that
\begin{align*}
 \Phi_u(u,1)&=\lim_{\epsilon\to 0} \frac{\Phi(u+\epsilon,1)-\Phi(u,1)}{\epsilon}=
    \lim_{\epsilon\to 0} \frac{u+\epsilon-u}{\epsilon}= 1,\\
 \Phi_u(u,0)&=\lim_{\epsilon\to 0} \frac{\Phi(u+\epsilon,0)-\Phi(u,0)}{\epsilon}=
    \lim_{\epsilon\to 0} \frac{0}{\epsilon}= 0,\\
\end{align*}
and analogously $\Phi_v(1,v)=u$ and $\Phi_v(0,v)=0$ for all $v\in [0,1]$. Consequently, we obtain for each $u\in [0,1]$ that
\begin{align}\label{eq.part-int=1}
  \int_0^1 \phi(u,v)\,dv
  =\int_0^1 \Phi_{u,v}(u,v)\,dv
  =   \Phi_u(u,1)-\Phi_u(u,0)=1.
\end{align}
For fixed $k\in\N\setminus\{1\}$ it follows from its definition in \eqref{eq.def-consistent-2} that $C_k(u_1,\dots, u_k)=0$
if at least one of the arguments $u_j$ equals $0$.
For each $u_j\in [0,1]$ and $j\in\{2,\dots, k\}$, Tonelli's theorem and \eqref{eq.part-int=1} imply  that
\begin{align}\label{eq.aux1}
 &C_k(1,\dots, 1,u_j,1,\dots, 1)\notag\\
 &=\int_0^{1}\bigg(\int_0^1 \phi_1(v_1,v_2)\bigg( \dots \int_0^{u_j} \phi_{j-1}(v_{j-1},v_j)\bigg(\cdots \int_0^1
 \phi_{k-1}(v_{k-1},v_k)\,dv_k\bigg)\cdots dv_j\bigg)\cdots \bigg)dv_1\notag\\
 &=\int_0^{1}\bigg(\int_0^1 \phi_1(v_1,v_2)\bigg( \dots \int_0^{u_j} \phi_{j-1}(v_{j-1},v_j) dv_j\bigg)\cdots \bigg)dv_1\notag\\
 &=\int_0^{u_j}\bigg(\int_0^1 \phi_{j-1}(v_{j-1},v_j)\bigg( \dots \int_0^1 \phi_1(v_1,v_2) dv_1\bigg)\cdots \bigg)dv_j\notag\\
&= \int_0^{u_j}\, dv_j\notag\\
&=u_j.
\end{align}
Define a function
\begin{align*}
\psi\colon [0,1]^k\to [0,1],\qquad \psi_k(v_1,\dots, v_k)=\phi_1(v_1,v_2)\cdot\ldots\cdot \phi_{k-1}(v_{k-1},v_k).
\end{align*}
By taking $u_j=1$ in \eqref{eq.aux1} we obtain
\begin{align*}
  \int_{\R^{k}}\psi_k(v_1,\dots, v_k)\,dv_k\cdots dv_1=1,
\end{align*}
which shows that $\Psi_k$ is a probability density function on $[0,1]^k$. Since $C_k$ is the cumulative distribution function of $\psi_k$, we have established that $C_k$ is a copula in $\R^k$.

It follows from \eqref{eq.part-int=1} for every $u_1,\dots, u_{k} \in [0,1]$ and $k\in \N$ that
\begin{align*}
&C_{k+1}(u_1,\dots, u_{k},1)\\
&= \int_0^{u_1}\cdots \int_0^{u_{k}}
   \phi_1(v_1,v_2)\cdots \phi_{k-1}(v_{k-1},v_{k})\int_0^1\phi_{k}(v_{k},v_{k+1})\,dv_{k+1}\,dv_{k}\cdots dv_1\\
&= \int_0^{u_1}\cdots \int_0^{u_{k}}
  \phi_1(v_1,v_2)\cdots \phi_{k-1}(v_{k-1},v_{k})\,dv_{k} \cdots dv_1\\
&= C_{k}(u_1,\dots, u_{k}),
\end{align*}
which completes the proof.
\end{proof}
\begin{example}\label{ex.2-normal}
The density $\phi_k\colon [0,1]^2\to\R$ of the Gaussian copula in $\R^2$ with correlation parameter
$\rho_k\in (-1,1)$ for every $k\in\N$ is given by
\begin{align*}
\phi_k(v_1,v_2)
&=\frac{1}{\sqrt{1-\rho_k^2}}
\exp\left(  -\frac{\rho_k}{2(1-\rho_k^2)}
\begin{pmatrix}G^{-1}(v_1) \\ G^{-1}(v_2) \end{pmatrix}^T
\begin{pmatrix}   \rho_k &  -1 \\ -1 & \rho_k    \end{pmatrix}
\begin{pmatrix}G^{-1}(v_1) \\ G^{-1}(v_2) \end{pmatrix}
 \right)\\
&=\frac{1}{\sqrt{1-\rho_k^2}}
\exp\left(  -\frac{ \rho_k^2 (G^{-1}(v_1))^2 - 2\rho_k G^{-1}(v_1)G^{-1}(v_2)+ \rho_k^2
  (G^{-1}(v_2))^2}{2(1-\rho_k^2)}
\right),
\end{align*}
where $G\colon\R\to [0,1]$ is the cumulative distribution function of the standard normal distribution.
\end{example}

\begin{lemma}
Let $\gamma$ be a cylindrical probability measure on $H$  and let $F_{i,i+1\dots ,i+j}$ denote the cumulative distribution function of $\gamma\circ \pi_{e_i,\dots, e_{i+j}}^{-1}$ for every $i\in\N$ and $j\in \N_0$. Then
\begin{align}\label{eq.def-copula-dis}
  C_k(u_1,\dots, u_k):=F_{1,\dots, k}\Big(F^{-1}_1(u_1),\dots, F^{-1}_k(u_k)\Big)
  \qquad\text{for }u_1,\dots, u_k\in [0,1],
\end{align}
defines a family $(C_k)_{k\in\N}$ of consistent copulas.
\end{lemma}
\begin{proof}
It is well known that for each $k\in\N$ the relation \eqref{eq.def-copula-dis}
defines a copula  $C_k$ in $\R^k$. Let $(\Omega,\A,P)$ be a probability space and
 $\Gamma\colon H\to L_P^0(\Omega,\R)$ be a cylindrical random variable with cylindrical distribution $\gamma$. It follows for each $u_1,\dots ,u_{k}\in [0,1]$ and $k\in\N$ that
\begin{align*}
 C_{k+1}(u_1,\dots, u_{k},1)
 &=\lim_{x\to\infty} F_{1,\dots, k+1}\Big(F^{-1}_1(u_1),\dots, F^{-1}_{k}(u_{k}),x \Big)\\
 &=\lim_{x\to\infty} P\Big( \Gamma e_1\le F^{-1}_1(u_1),\dots, \Gamma e_k\le F^{-1}_k(u_{k}), \Gamma e_{k+1}\le x\Big)\\
 &=P\Big( \Gamma e_1\le F^{-1}_1(u_1),\dots, \Gamma e_k\le F^{-1}_k(u_{k}), \Gamma e_{k+1}\le \infty\Big)\\
 &=P\Big( \Gamma e_1\le F^{-1}_1(u_1),\dots, \Gamma e_k\le F^{-1}_{k}(u_{k})\Big)\\
 &=  C_k(u_1,\dots, u_{k}),
\end{align*}
which completes the proof.
\end{proof}
\begin{example}
  Let $\gamma$ be a centred, Gaussian cylindrical distribution with covariance operator $Q\in \L(H,H)$, i.e.\ $Q$ is a positive and symmetric linear operator but not  necessarily trace-class. Then $\gamma\circ \pi_{e_i,\dots, e_{i+j}}^{-1}$ for $i\in\N$ and $j\in \N_0$
  is a Normal distribution on $\Borel(\R^{j+ 1})$ with expectation $0$
  and covariance matrix $\big( \scapro{Qe_m}{e_n}\big)_{m,n=i,\dots, i+j}$.
  In particular, if $\gamma$ is the canonical Gaussian cylindrical distribution on $H$   it follows that $Q=\Id$ and that $F_{i,i+1\dots ,i+j}$ is the cumulative distribution function of the standard normal distribution on $\Borel(\R^{j+1})$. This example will show later that the converse of Sklar's theorem is not true in Hilbert spaces.
\end{example}

\section{Copulas and the product space}\label{se.product}

In the following we are concerned with the converse implication of Sklar's theorem. For this purpose, recall the isomorphism $\ii$ defined in \eqref{eq.H-and-l2} which maps the Hilbert space $H$
to the space $\ell^2$ of square summable sequences.
%\begin{align*}
%\ell^2:=\Big\{ (x_k)_{k\in\N}:\, x_k\in\R, \sum_{k=1}^\infty x^2_k<\infty\Big\}.
%\end{align*}
The space $\ell^2$ is equipped with the standard inner product
$\scapro{\cdot}{\cdot}$, induced norm $\norm{\cdot}_{\ell^2}$ and with the
standard orthonormal basis $(f_k)_{k\in\N}$.

Define the product space $\R^{\infty}$ as the space of all functions from
$\N$ to $\R$. Thus, $\R^\infty $ can be identified by
\begin{align*}
\R^\infty:= \bigtimes_{k=1}^\infty \R:= \big\{(x_k)_{k\in\N}:\, x_k\in\R\big\}.
\end{align*}
The canonical projections are defined
for $i, j\in\N$ and $i\le j$ by
\begin{align*}
\kappa_{i,i+1,\dots, j}\colon \R^\infty \to \bigtimes_{k=i}^j \R,\qquad
\kappa_{i,i+1,\dots, j}\big((x_k)_{k\in\N}\big)=\big( x_i,x_{i+1},\dots, x_j\big).
\end{align*}
 The product space $\R^\infty$ is
equipped with the product $\sigma$-algebra $\Borel^\infty$ which is defined by
\begin{align*}
\Borel^\infty:=\bigotimes_{k=1}^\infty \Borel(\R):=\sigma\big(\big\{\{ (x_k)_{k\in\N}\in\R^\infty:\, x_j\in B\},\, B\in\Borel(\R), j\in\N \big\}\Big);
\end{align*}
see for example \cite[page 485]{Billingsley}.

\begin{proposition}\label{pro.product}
For every family $(C_k)_{k\in\N}$ of consistent copulas  and $\mu\in \M^{\otimes}(H)$ there exists
a unique probability measure $\nu$ on $(\R^\infty,\Borel^\infty)$ satisfying
 for each $I_1,\dots, I_k\in \I$ and $k\in\N$:
 \begin{align}\label{eq.mu-nu}
\nu\circ \kappa_{1,\ldots ,k} ^{-1}\Big(I_1\times \,\cdots\,\times  I_k\Big)=C_k\Big(\mu_1(I_1),\ldots, \mu_k(I_k)\Big),
\end{align}
where $\mu_j:=\mu\circ \pi_{e_j}^{-1}$.
\end{proposition}
\begin{proof}
Sklar's Theorem in $\R^k$ implies for each $k\in\N$  that there exists
a probability measure $\nu_{1,2,\dots, k}\in \M(\R^k)$ satisfying  for each $I_1,\dots, I_k\in \R$:
\begin{align}\label{eq.finite-Sklar-marginals}
  \nu_{1,2,\dots, k}\Big(I_1\times \,\cdots\,\times  I_k\Big)=C_k\Big(\mu_1(I_1),\ldots, \mu_k(I_k)\Big).
\end{align}
In order to apply Kolmogorov's consistency theorem we define a family $\{\nu_{t_1,\dots, t_k}:\, t_i\in\N \text{ pairwise disjoint}, k\in\N\}$ of probability measures in the following way: if $t_1<t_2<\dots < t_k$ and $B_{t_1},\dots, B_{t_k}\in\Borel(\R)$ then
\begin{align*}
\nu_{t_1,\dots, t_k}\big(B_{t_1}\times\cdots \times B_{t_k}\big)
:&=\nu_{1,2,\dots, t_1-1,t_1,t_1+1,\dots, t_k}\big(\R\times\dots \times \R\times B_{t_1}\times \R\times \dots \times B_{t_k}\big).
\end{align*}
This set function $\nu_{t_1,\dots, t_k}$ is uniquely extended to a measure $\nu_{t_1,\dots, t_k}$ in $\M(\R^k)$.
For arbitrary but mutually disjoint $t_1,\dots, t_k \in \N$ let $\chi$ denote the permutation of $\{t_1,\dots, t_k\}$ such that $\chi(t_1)<\dots <\chi(t_k)$ and
define
\begin{align*}
\nu_{t_1,\dots, t_k}\big(B_{t_1}\times\cdots \times B_{t_k}\big):=
\nu_{\chi(t_1),\dots, \chi(t_k)}\big(B_{\chi(t_1)}\times\cdots \times B_{\chi(t_k)}\big).
\end{align*}
Again, this set function $\nu_{t_1,\dots, t_k}$ is uniquely extended to a measure $\nu_{t_1,\dots, t_k}$ in $\M(\R^k)$.
The consistency of the copulas imply
for every $B_1,\dots, B_k\in\Borel(\R)$ and $k\in\N$ that
\begin{align*}
\nu_{1,2,\dots, k,k+1}(B_1\times \,\dots\,\times B_k\times \R)
= \nu_{1,2,\dots, k}(B_1\times \,\dots\,\times B_k).
\end{align*}
Consequently, we obtain  for $t_1<\dots <t_k$  that
\begin{align*}
&\nu_{t_1,\dots, t_k}\big(B_{t_1}\times\cdots \times B_{t_{k-1}}\times\R\big)\\
&=\nu_{1,2,\dots,t_{k-1}-1, t_{k-1},t_{k-1}+1,\dots, t_k}\big(\R\times\dots \times \R\times B_{t_{k-1}}\times \R\times \dots \times\R\big)\\
&=\nu_{1,2,\dots,t_{k-1}-1, t_{k-1}}\big(\R\times\dots \times \R\times B_{t_{k-1}})\\
&=\nu_{t_1,\dots, t_{k-1}}\big(B_{t_1}\times\cdots \times B_{t_{k-1}}\big).
\end{align*}
It follows that the family $\{\nu_{t_1,\dots, t_k}:\,t_k\in\N\text{ pairwise disjoint},k\in\N\}$
satisfies the conditions in Kolmogorov's consistency theorem which implies that there exists a unique probability measure $\nu$ on the measurable space  $(\R^\infty,\Borel^\infty)$ such that
the coordinate process $(\kappa_k:\, k\in\N)$ has $\nu_{t_1,\dots, t_n}$
 as its finite-dimensional distributions, i.e.\
\begin{align*}
\nu\circ \kappa_{t_1,\dots, t_k}^{-1}(B)
=\nu_{t_1,\dots, t_k}(B)
\qquad\text{for all }B\in\Borel(\R^k),\, t_1,\dots,t_k\in\N;
\end{align*}
see \cite[Th. 36.1]{Billingsley}. The uniqueness of the probability measure $\nu$ follows from the fact that the algebra generated by cylindrical sets is
a Dynkin system; see \cite[p.60]{Khoshnevisan}.
 \end{proof}

Similarly to the finite dimensional situation, Proposition \ref{pro.product} yields that every family $C=(C_k)_{k\in\N}$ of consistent copulas defines a mapping
\begin{align*}
\hat{C}\colon \M^{\otimes}(H)\to
 \M(\R^\infty),\qquad \mu\mapsto \hat{C}(\mu),
\end{align*}
where $\hat{C}(\mu)$ denotes the probability measure $\nu$ on $(\R^\infty,\Borel^\infty)$ satisfying \eqref{eq.mu-nu}. However, in contrast to the finite dimensional case the image $\hat{C}(\mu)$ is not necessarily a probability measure in the space $\ell^2$, and thus does not
result in a probability measure on the Hilbert space $H$. However, we obtain the following result:
\begin{theorem}\label{th.l2=1}
For a family $C=(C_k)_{k\in\N}$ of consistent copulas and $\mu\in \M^{\otimes}(H)$
 the following are equivalent:
\begin{enumerate}
\item[{\rm (1)}] $\hat{C}(\mu)$ is a  probability measure in $\M(\ell^2)$;
\item[{\rm (2)}] $\hat{C}(\mu)(\ell^2)=1$.
\end{enumerate}
\end{theorem}
\begin{proof}
Due to Proposition \ref{pro.product}, we only have to show that $\Borel(\ell^2)\subseteq \Borel^\infty$.
Recall that the product $\sigma$-algebra $\Borel^\infty$ is generated by the system
of sets
\begin{align*}
{\mathcal C}:=
\Big\{ \big\{ (x_k)_{k\in\N}\in\R^\infty:\, x_j\in B\big\}, \, B\in\Borel(\R), j\in\N    \Big\},
\end{align*}
i.e.\ we have $\Borel^\infty=\sigma({\mathcal C})$.
Define the system of sets
\begin{align*}
{\mathcal D}:= {\mathcal C}\cap \ell^2=\Big\{ \big\{ (x_k)_{k\in\N}\in\ell^2:\, x_j\in B\big\}, \, B\in\Borel(\R), j\in\N    \Big\}.
\end{align*}
The set ${\mathcal D}$ is separating for $\ell^2$, as for $(a_k)_{k\in\N}$, $(b_k)_{k\in\N}\in \ell^2$ with $a\neq b$ choose
$j\in\N$ such that $a_j\neq b_j$. Then with
$Z:=\big\{ (x_k)_{k\in\N}\in\ell^2:\, x_j\in \{a_j\}\big\}$,
which is an element of $\mathcal D$,
 it follows that
\begin{align*}
1= \1_{Z}\big((a_k)_{k\in\N}\big)\neq \1_{Z}\big((b_k)_{k\in\N}\big)=0.
\end{align*}
It follows from \cite[Pro.I.1.4]{Vaketal} that
$\Borel(\ell^2)= \sigma({\mathcal D})$. The representation
\begin{align*}
 \ell^2= \bigcup_{n=1}^\infty \bigcap_{m=1}^\infty \Big\{ (x_k)_{k\in\N}\in\R^\infty:\sum_{k=1}^m x_k^2 \le n\Big\}
\end{align*}
yields that $\ell^2\in \Borel^\infty$. Thus, we have
${\mathcal D}\subseteq\Borel^\infty$ which completes the proof.
\end{proof}

\begin{corollary}
Let $(C_k)_{k\in\N}$ be a family of consistent copulas.
If $\mu\in \M^{\otimes}\cap \M^2(H)$  then there exists a
 probability measure $\nu$ in $\M^2(H)$ satisfying
 for each $I_1,\dots, I_k\in \I$ and $k\in\N$:
 \begin{align*}
\nu\circ \pi_{e_1,\ldots ,e_k} ^{-1}\Big(I_1\times \,\cdots\,\times  I_k\Big)=C_k\Big(\mu_1(I_1),\ldots, \mu_k(I_k)\Big),
\end{align*}
where $\mu_j:=\mu\circ \pi_{e_j}^{-1}$.
\end{corollary}
\begin{proof}
Let  $\hat{\nu}:=\hat{C}(\mu)$ denote the probability measure on $(\R^\infty,\Borel^\infty)$ satisfying \eqref{eq.mu-nu} according to Proposition \ref{pro.product}, and let $(\kappa_k:\, k\in\N)$ be the canonical coordinate process on $(\R^\infty,\Borel^\infty)$. By denoting $\hat{\nu}_j:=\hat{\nu}\circ \kappa_j^{-1}$ it follows from \eqref{eq.mu-nu}
for each $I\in\I$  that
\begin{align*}
\hat{\nu}_j(I)
=\hat{\nu}\circ\kappa_{1,\dots, j}^{-1}(\R\times\,\cdots\,\times \R\times I)
= C_k(1,\dots, 1, \mu_j(I))
=\mu_j(I),
\end{align*}
which yields $\hat{\nu}_j=\mu_j$. Thus, we obtain
\begin{align*}
E_{\hat{\nu}}\big[\norm{\kappa}_{\ell^2}^2\big]
= \sum_{k=1}^\infty E_{\hat{\nu}}\big[\kappa_k^2\big]
&=\sum_{k=1}^\infty \int_{\R} x^2\, \hat{\nu}_k(dx)\\
&=\sum_{k=1}^\infty \int_{\R} x^2\, \mu_k(dx)\\
&=\sum_{k=1}^\infty \int_{H} \scapro{h}{e_k}^2 \mu(dh)
=\int_H \norm{h}^2\,\mu(dh)<\infty.
\end{align*}
It follows that $\norm{\kappa}_{\ell^2}<\infty$ $\hat{\nu}$-a.s., which yields $\hat{\nu}(\ell^2)=1$.
Theorem \ref{th.l2=1} implies that $\hat{\nu}$ is a probability measure in $\M(\ell^2)$.

Denote the inverse of the isomorphism $\ii\colon H\to \ell^2$ defined in
\eqref{eq.H-and-l2} by
$\jj$ and define a probability measure  $\nu:=\hat{\nu}\circ \jj^{-1}$.
Then we have $\nu\in \M(H)$ and
\begin{align*}
\int_H \norm{h}^2\,\nu(dh)
=\sum_{k=1}^\infty\int_{\R} x^2  (\nu\circ \pi_{e_k}^{-1})(dx)
\sum_{k=1}^\infty \int_{\R} x^2\, \hat{\nu}_k(dx)
=\int_H \norm{h}^2\,\mu(dh)<\infty.
\end{align*}
Equality \eqref{eq.finite-Sklar-marginals}
implies for every $I_1,\dots, I_k\in\I$ and $\in\N$ that
\begin{align*}
&\nu\circ\pi_{e_1,\dots, e_k}^{-1}\big(I_1\times\, \cdots \,\times I_k\big)\\
&\qquad= \hat{\nu}\Big(\Big\{(x_k)_{k\in\N}\in \ell^2:\, \pi_{e_1,\dots, e_k}\big(j((x_k)_{k\in\N})\big)
\in I_1\times\, \cdots \,\times I_k\Big\}\Big)\\
&\qquad =\hat{\nu}\Big(\Big\{(x_k)_{k\in\N}\in \ell^2:\,
 (x_1,\dots, x_k)\in I_1\times\, \cdots \,\times I_k\Big\}\Big)\\
&\qquad=\hat{\nu}\circ \kappa_{1,\dots, k}^{-1}\Big( I_1\times\, \cdots \,\times I_k\Big)\\
&\qquad= C_k\big(\mu_1(I_1),\dots, \mu_k(I_k)\big),
\end{align*}
which completes the proof.
\end{proof}

\section{Copulas with densities}\label{se.densities}

Most of the examples of copulas in finite dimensional spaces have a density.
For a family of consistent copulas, the following property guarantees
similar results in the infinite dimensional setting as in Euclidean spaces.
\begin{definition}
Let $(C_k)_{k\in\N}$ be a family  of consistent copulas $C_k$ with densities $c_k\colon [0,1]^k\to \Rp$. The family $(c_k)_{k\in\N}$ of densities is called {\em uniformly integrable} if
\begin{align*}
\lim_{r\to\infty}\sup_{k\in\N}
 \int_{\{u\in [0,1]^k:\, c_k(u)\ge r\}}
  c_k(u)\,du=0.
\end{align*}
\end{definition}
We need the following definition from asymptotic statistics.
\begin{definition}
Let $(S_k,\S_k)$ be a measurable space with two probability measures $P_k$ and $Q_k$
for every $k\in\N$. Then the sequence $(Q_k:\,k\in\N)$ is {\em contigiuous with respect
to } $(P_k:\,k\in\N)$ if  for every sequence $(A_k)_{k\in\N}$ of
sets $A_k\in\S_k$ we have:
\begin{align*}
\lim_{k\to\infty} P_k(A_k)=0
\;\Rightarrow \; \lim_{k\to\infty} Q_k(A_k)=0.
\end{align*}
This is denoted by $(Q_k:\,k\in\N)
\lhd (P_k:\,k\in\N)$.
\end{definition}

\begin{theorem}\label{th.uni-form}
 Let $(C_k)_{k\in\N}$ be a  family of consistent copulas $C_k$ with  densities $c_k$.
Then for every $\mu\in \M^{\otimes}\cap \M^c(H)$ the following are
equivalent:
\begin{enumerate}
\item[{\rm (1)}] $(c_k)_{k\in\N}$ is uniformly integrable.
\item[{\rm (2)}] there exists a measure $\nu\in \M(H)$ satisfying:
 \begin{enumerate}
 \item[{\rm (a)}] $(\nu\circ \pi_{e_1,\dots, e_k}^{-1}:k\in\N) \lhd
     (\mu\circ \pi_{e_1,\dots, e_k}^{-1}:k\in\N)$.
\item[{\rm (b)}]
     for each $I_1,\dots, I_k\in \I$ and $k\in\N$ we have:
\begin{align}\label{eq.th-eq}
\nu\circ \pi_{e_1,\ldots ,e_k} ^{-1}\Big(I_1\times \,\cdots\,\times  I_k\Big)=C_k\Big(\mu_1\big(I_1)\big),\ldots, \mu_k\big(I_k\big)\Big),
\end{align}
where $\mu_j:=\mu\circ \pi_{e_j}^{-1}$.
 \end{enumerate}
\end{enumerate}
In this situation the measure $\nu$ is absolutely continuous with respect to $\mu$.
\end{theorem}

\begin{proof}
We begin with deriving a few formulas which are true without assuming (1) or (2).  Denote by $F_j$ the cumulative distribution function of $\mu_j$ and  define the function
\begin{align*}
 g_k\colon H\to [0,1]^k,\qquad  g_k(h)=\big(F_1(\scapro{h}{e_1}),\dots, F_k(\scapro{h}{e_k})\big).
\end{align*}
The probability integral transform\footnote{I think that here we need continuous cdf, i.e.\ $\mu\in \M^c$} guarantees for every $B_1,\dots, B_k\in\Borel(\R)$ that
\begin{align*}
 \mu\circ g_k^{-1}(B_1\times\, \cdots \,\times B_k)
 &=\mu\big(h\in H :
  F_1(\scapro{h}{e_1})\in B_1,\dots, F_k(\scapro{h}{e_k})\in B_k\big)\notag\\
 &=\prod_{j=1}^k \mu_j(F_j^{-1}(B_j))\notag\\
&= \lambda_k\big((B_1\times\cdots\times B_k)\cap [0,1]^k\big),
\end{align*}
where $\lambda_k$ denotes the Lebesgue measure on $\Borel(\R^k)$. Consequently, we can conclude:
\begin{align}\label{eq.mu-and-lambda}
 \mu\circ g_k^{-1}=\lambda_k|_{[0,1]^k}
 \qquad\text{for all }k\in\N.
\end{align}
For each $k\in\N$ define a  measure on $\Borel(H)$ by
\begin{align*}
\nu_{k}\colon \Borel(H)\to [0,1],\qquad
\nu_{k}(B) =\int_{B} c_k\big(g_k(h))\big)\,\mu(dh).
\end{align*}
It is a probability measure since \eqref{eq.mu-and-lambda} guarantees
\begin{align*}
\nu_k(H)=\int_H c_k(g_k(h))\,\mu(dh)
=\int_{[0,1]^k} c_k(x)\, dx=1.
\end{align*}
We show that each measure $\nu_k$ satisfies \eqref{eq.th-eq}. For this purpose, define for every $j\in\N$ the Borel set $A_j\in\Borel(\R)$  where the cdf $F_j$ is constant:
\begin{align*}
 A_j:=\bigcup_{q\in\Q} \big\{x\in \R:\, F_j(x)
 =F_j(x+q)\big\}.
\end{align*}
It follows that
\begin{align*}
\mu_j\big(A_j)\le
   \sum_{q\in\Q}\mu_j\big(\big\{x\in \R:\, F_j(x)
    =F_j(x+q)\big\}\big)
    =0.
\end{align*}
Define for every $k\in\N$ and $j\in\{1,\dots, k\}$ the Borel set
\begin{align*}
B_j^k:=\pi_{e_1,\dots, e_k}^{-1}(\R\times\cdots\times\R\times A_j\times\R\times\cdots \times\R)
  \in \Borel(H).
\end{align*}
It follows that $B_j^k$ is a $\mu$-nullset as
\begin{align*}
\mu(B_j^k)&=\mu\circ \pi_{e_1,\dots,e_k}^{-1}(\R\times\cdots\times\R\times A_j\times\R\times\cdots \times\R)\\
&=\mu_1(\R)\cdot\ldots\cdot \mu_{j-1}(\R)
\mu_j(A_j)\mu_{j+1}(\R)\cdot\ldots\cdot\mu_k(\R)
=0.
\end{align*}
Consequently, the set $B^k:=B_1^k\cup\dots \cup B_k^k$ is also a $\mu$-nullset. For each $u_1,\dots, u_k\in [0,1]$ define $I(u_j)=(-\infty,u_j]$.  As every cdf $F_j$ for $j=1,\dots, k$ is strictly increasing on $(B^k)^c$ we conclude
\begin{align*}
&\big\{h\in (B^k)^c:\,
\scapro{h}{e_1}\le u_1,\dots, \scapro{h}{e_k}\le u_k\big\}\\
&\qquad \qquad= \big\{h\in (B^k)^c:
F_1(\scapro{h}{e_1})\le F(u_1),\dots,F_k(\scapro{h}{e_k})\le F_k(u_k)\big\}.
\end{align*}
Together with \eqref{eq.mu-and-lambda}, it follows  that
\begin{align}\label{eq.nu-and-C}
 &\big(\nu_k\circ \pi^{-1}_{e_1,\dots, e_k}\big)\big(I(u_1)\times \ldots \times I(u_k)\big)\notag\\
 &\qquad
=\nu_k \big(\big\{h\in H: \pi_{e_1,\dots, e_k}(h)
 \in I(u_1)\times \cdots \times I(u_k)\big\}\big)\notag\\
 &\qquad=\int 1_{I(u_1)\times \cdots \times I(u_k)}\big(\pi_{e_1,\dots,e_k}(h)\big) c_k(g_k(h))\,\mu(dh)\notag\\
 &\qquad=\int_{(B^k)^c} 1_{I(u_1)\times \cdots \times I(u_k)}\big(\pi_{e_1,\dots,e_k}(h)\big) c_k(g_k(h))\,\mu(dh)\notag\\
  &\qquad=\int_{(B^k)^c} 1_{I(F_1(u_1))\times\dots\times I(F_k(u_k))}\big(g_k(h)\big) c_k(g_k(h))\,\mu(dh)\notag\\
  &\qquad=\int 1_{I(F_1(u_1))\times\cdots\times I(F_k(u_k))}\big(g_k(h)\big) c_k(g_k(h))\,\mu(dh)\notag\\
    &\qquad=\int 1_{I(F_1(u_1))\times\cdots\times I(F_k(u_k))}\big((s_1,\dots,s_k)\big) c_k(s_1,\dots,s_k))\,(\mu\circ g_k^{-1})(ds_1,\dots, ds_k)\notag\\
&\qquad = C_k\big(F_1(u_1),\dots, F_k(u_k)\big).
\end{align}
In particular, relation \eqref{eq.nu-and-C} and the consistency of the copulas imply for every $\ell\ge k $ that
\begin{align*}
 &\big(\nu_\ell\circ \pi^{-1}_{e_1,\dots, e_k}\big)\big(I(u_1)\times \cdots \times I(u_k)\big)\\
 &\qquad
  =\big(\nu_\ell\circ \pi^{-1}_{e_1,\dots, e_\ell}\big)\big(I(u_1)\times \cdots \times I(u_k)
  \times\R\times\cdots \times \R\big)\\
 &\qquad
  = C_\ell\big(F_1(u_1),\dots, F_k(u_k),1,\dots, 1\big) \\
 &\qquad
  =C_k \big(F_1(u_1),\dots, F_k(u_k)\big)\notag\\
  &\qquad
  =\big(\nu_k\circ \pi^{-1}_{e_1,\dots, e_k}\big)\big(I(u_1)\times \cdots \times I(u_k)\big).
\end{align*}
Thus, we can conclude that
\begin{align}\label{eq.consistency-projection}
\nu_{\ell}\circ \pi_{e_1,\dots,e_k}^{-1}=\nu_k\circ\pi_{e_1,\dots,e_k}^{-1}
\qquad\text{for all } \ell\ge k.
\end{align}
Now we can establish that on the probability space $(H,\Borel(H),\mu)$ the random variables
\begin{align}\label{eq.def-martingales}
M_k\colon H\to\R,\qquad
M_k(h):=c_k(g_k(h))
\end{align}
form a martingale $(M_k:\,k\in\N)$ with respect to the filtration $(\F_k)_{k\in\N}$ defined by
\begin{align*}
\F_k:=\pi_{e_1,\dots, e_k}^{-1}\big(\Borel(\R^k)\big).
\end{align*}
Clearly, $M_k$ is $\F_k$-measurable and its definition yields for every $k\in\N$ that
\begin{align}\label{eq.M-and-nu}
E_\mu\big[\1_B M_k\big]= \nu_k(B)
\qquad\text{for every }B\in\Borel(H),
\end{align}
which shows $E_\mu[M_k]=1$.
For arbitrary $B\in \F_{k}$ there exists $A\in \Borel(\R^k)$ such that $B=\pi_{e_1,\dots, e_k}^{-1}(A)$. The equalities \eqref{eq.consistency-projection} and \eqref{eq.M-and-nu} imply
\begin{align*}
E_\mu\big[\1_B M_{k+1}\big]
= \nu_{k+1}\circ \pi_{e_1,\dots, e_k}^{-1}(A)
= \nu_{k}\circ \pi_{e_1,\dots, e_k}^{-1}(A)
= E_\mu\big[\1_B M_k\big],
\end{align*}
which establishes that $(M_k:\, k\in\N)$ is a martingale. As $(M_k:\, k\in\N)$ is a non-negative martingale, there exists a random variable $M$ on $(H,\Borel(H),\mu)$ with $E_\mu[\abs{M}]<\infty$ such that $M_k\to M$ $\mu$-a.s.;
see \cite[Th. 35.4]{Billingsley}.

Define the probability measures $p_k:=\mu\circ \pi_{e_1,\dots ,e_k}^{-1}$ and $q_k:=\nu_k\circ \pi_{e_1,\dots, e_k}^{-1}$ for every $k\in\N$. By defining  the function
\begin{align*}
 F^k\colon \R^k \to [0,1]^k,
 \qquad F^k (x_1,\dots, x_k)=
 \big(F_1(x_1),\dots, F_k(x_k)\big),
\end{align*}
we obtain for every $A\in\Borel(\R^k)$ that
\begin{align*}
q_k(A)&=\int_H \1_{\pi_{e_1,\dots,e_k}^{-1}(A)}(h) c_k(g_k(h))\,\mu(dh)\\
&=\int_H \1_{A}(\pi_{e_1,\dots,e_k}(h)) c_k(F^k(\pi_{e_1,\dots,e_k}(h)))\,\mu(dh)\\
& =\int_{\R^k} \1_{A}(x)c_k(F^k(x))\,p_k(dx).
\end{align*}
Thus, the Radon-Nikodym density of $q_k$ with respect to $p_k$ is given by
\begin{align}\label{eq.Radon-Nikodym}
r_k:=\frac{dq_k}{dp_k}=c_k\circ F^k.
\end{align}
The Radon-Nikodym density $r_k$ satisfies for every continuous and bounded function
$f\colon \R\to\R$ that
\begin{align}\label{eq.R-N-and-M}
E_{p_{k}}\big[ f(r_{k})\big]
=\int_{\R^{k}} f\big(c_{k}\big(F^{k}(x)\big)\big)\,p_{k}(dx)
=E_\mu\big[f(M_{k})\big].
\end{align}

(1)$\Rightarrow$ (2):
For given $\epsilon>0$ it follows from \eqref{eq.mu-and-lambda} and the uniform integrability of the densities that there exists $r>0$ such that we have for all $k\in\N$:
\begin{align}\label{eq.uniform-int}
\int_{\{h\in H: c_k(g_k(h))>r\}} c_k(g_k(h))\,\mu(dh)
&= \int_{\{x\in\R^k c_k(x)>r\}} c_k(x)\,dx
\le \epsilon.
\end{align}
Since $\mu$ is a Radon measure there exists a compact set $K\subseteq H$ such that $ \mu(K^c)\le \frac{\epsilon}{r}$.
Combining the two estimates  implies for all $k\in\N$ that
\begin{align*}
\nu_k(K^c)&=\int_{\{h\in K^c: c_k(g_k(h))>r\}} c_k(g_k(h))\,\mu(dh)
 + \int_{\{h\in K^c: c_k(g_k(h))\le r\}} c_k(g_k(h))\,\mu(dh)\\
& \le \epsilon + r \frac{\epsilon}{r}=2\epsilon.
\end{align*}
Prokhorov's Theorem implies that the sequence $(\nu_k)_{k\in\N}$ is relatively compact in $\M(H)$.

In the following we establish that  $(\nu_k)_{k\in\N}$ converges weakly to a measure  $\nu\in\M(H)$. For this purpose, let $(\nu_{k_n})_{n\in\N}$ be a sub-sequence converging weakly to a measure $\nu\in \M(H)$. It follows for fixed $\ell\in \N$ that $\nu_{k_n}\circ \pi_{e_1,\dots, e_\ell}^{-1}$ converges weakly to $\nu\circ \pi_{e_1,\dots, e_\ell}^{-1}$ in $\M(\R^\ell)$ for $n\to\infty$. On the other hand,
relation \eqref{eq.consistency-projection} implies for every $n\in\N$ with $k_n\ge \ell$ that
\begin{align*}
\nu_{k_n}\circ \pi_{e_1,\dots,e_\ell}^{-1}=\nu_\ell\circ\pi_{e_1,\dots,e_\ell}^{-1}.
\end{align*}
Consequently, it follows for every $\ell\in\N$ that
\begin{align}\label{eq.nu-eventually-const}
 \nu\circ \pi_{e_1,\dots, e_\ell}^{-1}
 =\nu_\ell\circ\pi_{e_1,\dots,e_\ell}^{-1}.
\end{align}
Relation \eqref{eq.nu-and-C} implies for every $I_1,\dots, I_\ell\in \I$ and $\ell\in\N$ that
 \begin{align*}
\nu\circ \pi_{e_1,\dots, e_\ell}^{-1}\Big(I_1\times \,\cdots\,\times  I_\ell\Big)
&=\nu_\ell\circ \pi_{e_1,\dots, e_\ell}^{-1}\Big(I_1\times \,\cdots\,\times  I_\ell\Big)\\
&=C_\ell\Big(\mu_1\big(I_1)\big),\ldots, \mu_\ell\big(I_\ell)\big)\Big).
\end{align*}
The uniqueness statement in Proposition \ref{pro.product} yields that $\nu\circ \ii^{-1}=\hat{C}(\mu)$. Thus, each weakly convergent subsequence of $(\nu_k)_{k\in\N}$ converges to the same measure $\nu\in \M(H)$ which satisfies \eqref{eq.th-eq}, and thus the proof of (b) is completed.

For establishing (a), define for each $k\in\N$ the probability measure $Q_k:=\nu\circ \pi_{e_1,\dots, e_k}^{-1}$. Equality \eqref{eq.nu-eventually-const} implies   that the Radon-Nikodym density of $Q_k$ with respect to $p_k$ is given by $r_k$ defined in \eqref{eq.Radon-Nikodym}. Assume that a subsequence $(r_{k_n})_{n\in\N}$ converges weakly to a Borel measure $r$, that is for each  continuous and bounded function $f\colon \R\to \R$ it follows
\begin{align*}
\lim_{n\to\infty }E_{p_{k_n}}\big[f(r_{k_n})\big]
=\int f(x)\,r(dx).
\end{align*}
Since equality \eqref{eq.R-N-and-M} and the weak convergence of the martingale $(M_k:\,k\in\N)$ imply
\begin{align}\label{eq.r-and-M}
E_{p_{k_n}}\big[ f(r_{k_n})\big]
=E_\mu\big[f(M_{k_n})\big]&\to\int_H f\big(M(h)\big)\,\mu(dh),
\end{align}
we conclude $r=\mu\circ M^{-1}$. Since Assumption (1) guarantees that  $(M_k:\, k\in\N)$ is a uniform integrable martingale by \eqref{eq.uniform-int} it follows that $M_k\to M$ in mean and thus
\begin{align*}
 1=E_\mu[M]=\int_{\R} x\,\big(\mu\circ M^{-1}\big)(dx).
\end{align*}
Theorem 3.2.10 in \cite{VaartWellner} implies (a).

(2)$\Rightarrow$(1). Define for each $k\in\N$ the probability measure $Q_k:=\nu\circ \pi_{e_1,\dots, e_k}^{-1}$. It follows from \eqref{eq.th-eq} and \eqref{eq.nu-and-C} that the Radon-Nikodym density of $Q_k$ with respect to $p_k$ is given by $r_k$ defined in \eqref{eq.Radon-Nikodym}. Assume that $(r_k)_{k\in\N}$ converges weakly to a measure $r$. As in \eqref{eq.r-and-M} it follows that $r=\mu\circ M^{-1}$.  Theorem 3.2.10 in \cite{VaartWellner} implies
\begin{align*}
1=\int_{\R} x\,r(dx)=\int_{\R} x\,\big(\mu\circ M^{-1}\big)(dx)
=\int_H M(h)\,\mu(dh)=E_\mu[M].
\end{align*}
It follows from \cite[Le.3.11]{Kallenberg}  that $(M_k:\,k\in\N)$ is uniformly integrable which yields (1).

It remains to show that $\nu$ is absolutely continuous with respect to $\mu$. For every $B\in \F_k$ for some $k\in\N$ there exists $A\in \Borel(\R^k)$ with $B=\pi_{e_1,\dots, e_k}^{-1}(A)$. It follows from \eqref{eq.M-and-nu}  and \eqref{eq.nu-eventually-const} that
\begin{align}\label{eq.nu-and-M}
\nu(B)=\nu\circ\pi_{e_1,\dots, e_k}^{-1}(A)
=\nu_k\circ\pi_{e_1,\dots, e_k}^{-1}(A)
= E_\mu\big[\1_B M_k\big].
\end{align}
As every set $B\in\Borel(H)$ obeys
\begin{align*}
B=\bigcap_{k=1}^\infty \pi_{e_1,\dots, e_k}^{-1}\big(\pi_{e_1,\dots, e_k}(B)\big),
\end{align*}
we obtain by defining $B_n:=\bigcap_{k=1}^n \pi_{e_1,\dots, e_k}^{-1}\big(\pi_{e_1,\dots, e_k}(B)\big)$ that $B_n\in\F_n$ and
\begin{align}\label{eq.nu-limit}
 \nu(B)=\lim_{n\to\infty} \nu(B_n).
\end{align}
On the other hand, since $\mu(B\setminus B_n)\to 0$ the uniform integrability of $(M_n:\,n\in\N)$ and the convergence of $M_n$  to $M$ in mean imply that
\begin{align*}
E_\mu\big[\abs{\1_{B_n}M_n- \1_B M}\big]
&\le E_\mu\big[ \1_{B\setminus B_n}M_n\big] + E_\mu\big[\1_B\abs{M_n-M}\big]\to 0.
\end{align*}
Consequently, it follows from \eqref{eq.nu-and-M} and \eqref{eq.nu-limit} that
\begin{align*}
\nu(B)= E_\mu[\1_B M] \qquad\text{for all }B\in \Borel(H),
\end{align*}
which completes the proof.
\end{proof}

The application of Theorem \ref{th.uni-form} requires the uniform integrability of
the densities of the copulas. The following result provides a simple criterion for copulas which are constructed by Lemma \ref{le.copulas-with-2}.
\begin{corollary}\label{co.uni-int-Hellinger}
Let $c_k\colon [0,1]^k\to \Rp$ be the density of a copula of the form
 \begin{align*}
 c_k(u_1,\dots, u_k)
 =\phi_1(u_1,u_2)\phi_2(u_2,u_3)\cdots \phi_{k-1}(u_{k-1},u_k)
 \quad\text{ for all }u_1,\dots, u_k\in [0,1],
 \end{align*}
where $\phi_j\colon [0,1]^2\to\Rp$ is the continuous density of a copula in $\R^2$
for all $j\in\N$ and $k\in\N\setminus\{1\}$ l. Then the following are equivalent:
\begin{enumerate}
\item[{\rm (1)}] the family $(c_k)_{k\in\N\setminus\{1\}}$ is uniformly integrable;
\item[{\rm (2)}] $\displaystyle
 \sum_{j=1}^\infty \left(1- \int_{[0,1]^2} \sqrt{\phi_j(u,v)}\, dudv\right)<\infty.$
\end{enumerate}

\end{corollary}
\begin{proof}
Choose an arbitrary probability measure $\mu$ in $\M^{\otimes}\cap \M^c(H)$ and
let the martingale $(M_k:\, k\in\N\setminus\{1\})$ on the probability space $(H,\Borel(H),\mu)$
be defined as in \eqref{eq.def-martingales}. Since we have for each $h\in H$ and $k\in\N\setminus\{1\}$ that
\begin{align*}
M_k(h)&=c_k\big( F_1(\scapro{h}{e_1}),\dots, F_k(\scapro{h}{e_k})\big)\\
& = \phi_1\big(F_1(\scapro{h}{e_1}),F_2(\scapro{h}{e_2})\big)
\cdots \phi_{k-1}\big(F_{k-1}(\scapro{h}{e_{k-1}}),F_k(\scapro{h}{e_k})\big)\\
&= M_{k-1}\big(h\big)\, \phi_{k-1}\big(F_{k-1}(\scapro{h}{e_{k-1}}),F_k(\scapro{h}{e_k})\big),
\end{align*}
we conclude $M_k= X_2\cdots X_k$ for every $k\in\N\setminus\{1\}$, where
\begin{align*}
X_j(h):=\phi_{j-1}\big(F_{j-1}(\scapro{h}{e_{j-1}}),F_{j}(\scapro{h}{e_{j}})\big)
\qquad \text{for all }h\in H,\, j=2,\dots , k.
\end{align*}
Kakutani's theorem for product martingales, see \cite[14.12]{Williams}, guarantees that $(M_k:\,k\in\N\setminus\{1\})$ is uniformly integrable if and only if
\begin{align*}
\sum_{j=2}^\infty \Big(1-E_\mu\left[X_j^{1/2}\right]\Big)<\infty.
\end{align*}
Recall that \eqref{eq.uniform-int} shows that the martingale  $(M_k:\,k\in\N\setminus\{1\})$ is uniformly integrable if and only if the family $(c_k)_{k\in\N\setminus\{1\}}$ is uniformly integrable.
Define the function
\begin{align*}
g_{j-1,j}\colon H\to [0,1]^2,
\qquad g_{j-1,j}(h)=\big(F_{j-1}(\scapro{h}{e_{j-1}}), F_j(\scapro{h}{e_j})\big).
\end{align*}
Then it follows from \eqref{eq.mu-and-lambda} that
\begin{align*}
E_\mu\left[ X_j^{1/2}\right]
&= \int_{H }
  \phi_{j-1}^{1/2}\big(F_{j-1}(\scapro{h}{e_{j-1}}),F_{j}(\scapro{h}{e_{j}})\big)
   \,\mu(dh)\\
  &= \int_{[0,1]^2 }
  \phi_{j-1}^{1/2}\big(u,v)\big)
   \,\big(\mu\circ g_{j-1,j}^{-1}\big)(dudv)\\
&= \int_{[0,1]^2}\phi_{j-1}^{1/2}(u,v)\, dudv,
\end{align*}
which completes the proof.
\end{proof}

Recall that the density $\phi_j\colon [0,1]^2\to \Rp$ of a copula is the
density function of a probability measure $\mu_j$ on $\Borel(\R^2)\cap [0,1]$.
The term $1- \int \sqrt{\phi_j(u,v)}\, dudv$  in Lemma \ref{co.uni-int-Hellinger} equals the squared Hellinger distance between $\mu_j$ and the product of its marginal distributions, which are
just the uniform distributions on the intervals $[0,1]$; see for example
\cite{Granger-etal}.

\begin{example} (Continues Example \ref{ex.2-normal}).
Let $\phi_k\colon [0,1]^2\to \Rp$ be the Gaussian copula in $\R^2$ with correlation parameter $\rho_k\in (-1,1)$ for $k\in\N$. Then one obtains
\begin{align*}
1- \int_{[0,1]^2} \sqrt{\phi_j(u,v)}\, dudv
= 1-\frac{\big(1-\rho_k^2\big)^{5/4}}{\big(1-\tfrac{\rho^2_k}{2}\big)^{3/2}};
\end{align*}
see \cite[Se.4.3]{MalevergneSornette}. In order that these terms are summable the sequence $(\rho_k)_{k\in\N}$ must converge to $0$ for $k\to\infty$,
and thus we can neglect the denominator.
The mean value theorem results in
\begin{align*}
1- \big(1-\rho_k^2\big)^{5/4}
= \tfrac{5}{2}(1-x_k^2)^{1/4} x_k\rho_k
\qquad\text{for some }x_k\in [0,\rho_k].
\end{align*}
Consequently, square summability of $(\rho_k)_{k\in\N}$,
i.e.\ $ \sum \rho_k^2<\infty$, implies by Lemma \ref{le.copulas-with-2},
that  the corresponding sequence $(c_k)_{k\in\N}$ of densities is uniformly integrable.
\end{example}

%\bibliographystyle{plain}
%\bibliography{lit-copulas}

\end{document}